\documentclass[preprint,12pt]{elsarticlearxiv}
\usepackage{graphicx}        
\usepackage{multicol}        
\usepackage{enumerate}
\usepackage[bottom]{footmisc}

\usepackage[english]{babel}
\usepackage{amssymb,amstext,amsmath,amsthm,mathabx}
\usepackage{pst-all}
\newtheorem{definition}{Definition} 
\newtheorem{remark}{Remark} 
\newtheorem{theorem}{Theorem} 
\newtheorem{corollary}{Corollary}

\newcommand{\no}[1]{\widebar{#1}}

\def\coh{\pi}
\def\F{\mathcal{F}}
\def\I{\mathcal{I}}

\def\D{\mathcal{D}}

\usepackage{amssymb}

\journal{}

\begin{document}

\begin{frontmatter}



\title{Probabilistic squares and hexagons of opposition\\ under coherence\tnoteref{mytitlenote}}
\tnotetext[mytitlenote]{This is a substantially extended version of a paper  (\cite{PfSa17}) presented at SMPS 2016 (Soft Methods in Probability and Statistics 2016) conference held in Rome in September 12--14, 2016.}
\author{Niki Pfeifer}
\address{Ludwig-Maximilians-University Munich, Germany\\ \texttt{niki.pfeifer@lmu.de}}
\author{Giuseppe Sanfilippo}
\address{University of Palermo, Italy\\ \texttt{giuseppe.sanfilippo@unipa.it}}
\address{}
\begin{abstract}
Various semantics  for studying the square of opposition and the hexagon of opposition have been proposed recently.
We   interpret \emph{sentences} by  imprecise (set-valued) probability assessments  on a finite sequence of  conditional events. We introduce the \emph{acceptability} of a sentence within coherence-based probability theory. 
We analyze the relations of the square and of the hexagon in terms of acceptability. Then, we show how to construct probabilistic versions of the square and of the hexagon of opposition by forming suitable tripartitions of the set of all coherent assessments. Finally, as an application, we present new versions of the  square and of the hexagon involving generalized quantifiers.
\end{abstract}


\end{frontmatter}


\section{Introduction}
There is a long history of investigations on the square of opposition spanning over two millenia \cite{beziau2014,sep-square}. A \emph{square of opposition}  represents logical key relations among basic sentence types in a diagrammatic way. 
The basic sentence types, traditionally denoted by $A$ (universal affirmative: ``Every $S$ is $P$''), $E$ (universal negative: ``No $S$ is $P$''), $I$ (particular affirmative: ``Some $S$ are $P$''), and $O$ (particular negative: ``Some $S$ are not $P$''), constitute the corners of the square. The diagonals and the sides of the square of opposition are formed by the following logical relations among the basic sentence types:
$A$ and $E$ are \emph{contraries} (i.e., they cannot both be true), $I$ and $O$ are \emph{subcontraries} (i.e., they cannot both be false), $A$ and $O$ as well as $E$ and $I$ are \emph{contradictories} (i.e., they cannot both be true and they cannot both be false), $I$ is a \emph{subaltern} of $A$ and $O$ is a \emph{subaltern} of $E$ (i.e., $A$ entails $I$ and $E$ entails $O$; for a visual representation see, e.g., Figure~\ref{fig:Qsquare} below, and cover the probabilities for seeing the traditional square of opposition). 
In the early 1950ies, the square of opposition was expanded to the \emph{hexagon of opposition},  by adding a sentence at the top and another one at the bottom of the square (see, e.g., Figure~\ref{FIG:hexagon}). Recently, the square of opposition as well as the hexagon of opposition and its extensions have been investigated from various semantic points of view (see, e.g.,  \cite{beziau12,beziau2014,Ciucci2015,DuboisPrade2012,Dubois2015,Dubois2017,gilio16,novak14,murinova2016,Murinova2016a}). 
In this paper we present a probabilistic analysis of the square of opposition under coherence, introduce the hexagon of opposition under coherence, and study the semantics of basic key relations among quantified statements.

After preliminary notions  (Section~\ref{SEC:PRELIM}), we introduce, based on g-coherence, a (probabilistic) notion of sentences and their acceptability and show how to construct squares of opposition under coherence from suitable tripartitions  (Section~\ref{SEC:SQUARE}). Then, we present an application  of our square to the study of generalized quantifiers (Section~\ref{SEC:GSQUARE}). In Section~\ref{SEC:HEX} we introduce the \emph{hexagon of opposition} under coherence.   Section~\ref{SEC:CONCL} concludes the paper by some remarks on future work.
\section{Preliminary Notions}
\label{SEC:PRELIM}
The
coherence-based approach to probability and to other uncertain measures has been adopted by many
authors (see, e.g.,
\cite{biazzo00,biazzo05,capotorti14,CaLS07,coletti12,coleti14FSS,Coletti2016,coletti02,coletti15possibilistic,gilio02,GOPS16,gilio13ins,GiSa14,pfeifer13b,pfeifer13,pfeifer09b});
we therefore recall only selected key features of coherence and its generalizations in this section.\\
An event $E$ is
a two-valued logical entity which can be either true or false.
The indicator of $E$ is a two-valued numerical quantity which is 1,
or 0, according to  whether the event $E$ is true, or false, respectively. We
use the same symbols for events and their indicators. We denote by
$\top$ the sure event (i.e., tautology or logical truth) and by $\bot$ the impossible event (i.e., contradiction or logical falsehood).  Moreover, given two events $E$ and $H$, we
denote by $E\land H$ (resp., $E \vee H$) conjunction  (resp., disjunction). To simplify
notation, we will use the product $EH$ to denote the conjunction $E\land H$,  which also denotes the indicator of $E\wedge H$. We denote by $\no{E}$ the negation of $E$. 
\\
Given two events $E$ and $H$, with $H\neq\bot$, the
\emph{conditional event} $E|H$ is defined as a three-valued logical
entity which is \emph{true} if $EH$ (i.e., $E\wedge H$) is true,
\emph{false} if $\no{E}H$ is true, and \emph{indetermined} (void) if $H$ is false (\cite[p. 307]{definetti74}).
In terms of 
the betting metaphor, if you assess $p(E|H)=p$, then you are willing
to pay (resp., to receive) an amount $p$ and  to receive (resp., to pay) 1, or 0, or $p$, according to
whether $EH$ is true, or $\no{E}H$ is true, or $\no{H}$ is true (bet
called off), respectively. 
For defining coherence,  consider  a real function $p : \; \mathcal{F}
\, \rightarrow \, \mathcal{R}$, where $\mathcal{F}$ is an arbitrary
family of conditional events. Consider a finite sub-family
$\mathcal{F}_n = (E_1|H_1, \ldots, E_n|H_n) \subseteq
\mathcal{F}$, and the vector $\mathcal{P}_n =(p_1, \ldots, p_n)$,
where $p_i = p(E_i|H_i) \, ,\;\; i = 1, \ldots, n$. We denote by
$\mathcal{H}_n$ the disjunction $H_1 \vee \cdots \vee H_n$. 
With the pair $(\mathcal{F}_n, \mathcal{P}_n$) we associate the random gain
${\mathcal{G}} = \sum_{i=1}^n s_iH_i(E_i - p_i)$,
where $s_1, \ldots, s_n$ are $n$ arbitrary real numbers. 
$\mathcal{G}$ represents  the net gain of  $n$ transactions, where for each transaction   its  meaning    is specified by the sign of $s_i$ (\emph{plus} for  buying  or \emph{minus} for selling) 
 and its scaling  is specified by the magnitude  of $s_i$. 
Denoting by $G_{\mathcal{H}_n}$ the set of values of $\mathcal{G}$ restricted to $\mathcal{H}_n$, we recall
\begin{definition}\label{COER-BET} \rm The function $p$ defined on $\mathcal{F}$ is called {\em coherent}
	if and only if, for every integer $n$, for every finite sub-family $\mathcal{F}_n$
	$\subseteq \mathcal{F}$ and for every $s_1, \ldots, s_n$, it holds that:
	$\min  G_{\mathcal{H}_n} \leq 0 \leq \max G_{\mathcal{H}_n}$.
\end{definition}
We say that $p$ is \emph{incoherent} if and only if $p$ is not coherent.\\
As shown by Definition \ref{COER-BET}, a probability assessment is coherent if and only if, in any finite combination of $n$ bets, it does not happen that the values in the set  $G_{\mathcal{H}_n}$ are all positive, or all negative ({\em no Dutch Book}).  Moreover, coherence of $p(E|H)$ requires that $p(E|H) \in [0,1]$ for every $E|H \in \mathcal{F}$. If $p$ on $\mathcal{F}$ is coherent, we call it a {\em conditional probability on $\mathcal{F}$} (see, e.g., \cite{Berti2002,coletti02,Regazzini1985}). 
Notice that, if $p$ is coherent, then $p$ also satisfies all the well known properties of finitely additive conditional probability (while the converse does not hold; see, e.g., \cite[Example 13]{coletti02}  or \cite[Example 8]{Gilio1995}).
\\
In what follows  $\mathcal{F}$ will  denote  finite sequence of conditional events. Let
$\mathcal{F}=(E_1|H_1,\ldots,E_n|H_n)$. We denote by $\mathcal{P}$
a (precise) probability assessment $\mathcal{P}=(p_1,\ldots,p_n)$ on $\mathcal
{F}$, where
$p_j=p(E_j|H_j)\in[0,1]$, $j=1,\dots,n$.  Moreover, we denote by $\Pi$
the set of \emph{all coherent precise} assessments on $\mathcal{F}$.  We recall that when there are no logical relations among the
events $E_1,H_1,\ldots, E_n,H_n$ involved in $\mathcal{F}$, that is
$E_1,H_1,\ldots, E_n,H_n$ are logically independent, then the set $\Pi$
associated with $\mathcal{F}$ is the whole unit hypercube $[0,1]^n$.
If there
are logical relations, then the set $\Pi$ \emph{could be} a strict
subset of $[0,1]^n$. As it is well known $\Pi\neq\emptyset$;
therefore, $\emptyset\neq\Pi\subseteq[0,1]^n$. 
If not stated otherwise, we do not make any assumptions concerning  logical independence. 
\begin{definition}\label{DEF:IA}
An \emph{imprecise, or set-valued, assessment} ${\mathcal{I}}$ on a family of conditional events $\mathcal{F}$
 is a (possibly empty) set of precise 
assessments $\mathcal{P}$ on $\mathcal{F}$.
\end{definition}
{Definition~\ref{DEF:IA}} states that
an \emph{imprecise (probability) assessment} ${\mathcal{I}}$ on a sequence 
of $n$ conditional events
$\mathcal{F}$  is just a (possibly empty)
subset of
$[0,1]^n$ (\cite{gilio98,gilio15ecsqaru,gilio16}). For instance, think about an agent (like Pythagoras)  who considers
only rational numbers to evaluate the probability of an event $E|H$. Pythagoras' evaluation can be represented by the imprecise assessment $\mathcal{I}=[0,1]\cap \mathbb{Q}$ on $E|H$. Moreover, a  constraint like $p(E|H)>0$ can be represented by the imprecise assessment $\mathcal{I}=]0,1]$ on $E|H$.\\
 Given an imprecise assessment ${\mathcal{I}}$ we denote by
$\no{\mathcal{I}}$ the
\emph{complementary imprecise assessment} of ${\mathcal{I}}$, i.e.
$\no{\mathcal{I}}=[0,1]^n\setminus{\mathcal{I}}$. We now recall the notions of g-coherence and total coherence in the general case of imprecise (in the sense of set-valued) probability assessments \cite{gilio16}.
\begin{definition}[g-coherence]\label{DEF:GCOH}
Given a sequence of $n$ conditional events $\mathcal{F}$. An imprecise
assessment ${\mathcal{I}}\subseteq[0,1]^n$ on $\mathcal{F}$ is \emph{g-coherent} iff there exists a coherent precise assessment $\mathcal{P}$
on $\mathcal{F}$
such that $\mathcal{P}\in{\mathcal{I}}$.
\end{definition}
\begin{definition}[t-coherence]
An imprecise assessment ${\mathcal{I}}$ on $\mathcal{F}$ is \emph
{totally coherent}
(t-coherent) iff the following two conditions are satisfied:
(i) ${\mathcal{I}}$ is non-empty; (ii) if $\mathcal{P}\in{\mathcal
{I}}$, then $\mathcal{P}$ is a
coherent precise assessment on $\mathcal{F}$.
\end{definition}
\begin{definition}[t-coherent part]
Given a sequence of $n$ conditional events $\mathcal{F}$. Let $\Pi$ be the set of all coherent assessments on $\mathcal{F}$. 
We denote by $\coh:\wp([0,1]^n)\rightarrow \wp(\Pi)$ the function defined by $\coh(\mathcal{I})=\Pi\cap \mathcal{I}$, for any imprecise assessment $\mathcal{I}\in \wp([0,1]^n)$. Moreover, for each subset $\mathcal{I}\in\wp([0,1]^n)$ we call  $\coh(\mathcal{I})$  the  \emph{t-coherent part} of $\mathcal{I}$.
\end{definition}
Of course,   if $\coh(\mathcal{I})\neq \emptyset$, then $\mathcal{I}$ is g-coherent and  $\coh(\mathcal{I})$ is t-coherent.
\section{From Imprecise Assessments to the Square of Opposition}
\label{SEC:SQUARE}
In this section we consider imprecise assessments on  a given sequence  $\mathcal{F}$ of $n$ conditional events. In our approach, a sentence $s$ is a pair
 $(\mathcal{F},\mathcal{I})$, where $\mathcal{I}\subseteq [0,1]^n$ is an imprecise assessment on $\mathcal{F}$. We introduce the following  equivalence relation under t-coherence:
\begin{definition}
\label{DEF:EQ}
Given two sentences $s_1:(\mathcal{F},\mathcal{I}_1)$ and $s_2:(\mathcal{F}, \mathcal{I}_2)$,   $s_1$ and $s_2$ are \emph{equivalent (under t-coherence)}, denoted by $s_1\equiv s_2$,  iff $\coh(\mathcal{I}_{1})=\coh(\mathcal{I}_{2})$. 
\end{definition}
\begin{definition}\label{DEF:OPER}
Given three sentences $s:(\mathcal{F},\mathcal{I})$, $s_1:(\mathcal{F},\mathcal{I}_1)$, and $s_2:(\mathcal{F},\mathcal{I}_2)$.  We define: $s_1\wedge s_2:(\mathcal{F},\mathcal{I}_1\cap \mathcal{I}_2)$ (conjunction); $s_1\vee s_2:(\mathcal{F},\mathcal{I}_1\cup \mathcal{I}_2)$  (disjunction); $\no{s}:(\mathcal{F},\no{\mathcal{I}})$, where $\no{\mathcal{I}}=[0,1]^n\setminus \mathcal{I}$ (negation).
\end{definition}
\begin{remark}\label{REM:1}
As the basic operations among sentences  are defined by  set-theoretical operations, they inherit the corresponding properties (including associativity, commutativity, De Morgan's law, etc.). 
Moreover, as $\coh(\mathcal{I}_1\cap \mathcal{I}_2)=\coh(\mathcal{I}_1)\cap \coh(\mathcal{I}_2)$, by setting $s_1^*=(\mathcal{F},\coh(\mathcal{I}_1))$, $s_2^*=(\mathcal{F},\coh(\mathcal{I}_2))$ and $(s_1\wedge s_2)^*:(\mathcal{F},\coh(\mathcal{I}_1\cap \mathcal{I}_2 ))$, it follows that 
 $(s_1\wedge s_2)\equiv (s_1\wedge s_2)^*\equiv s_1^*\wedge s_2^*$. Likewise, $s_1\vee s_2\equiv (s_1\vee s_2)^*\equiv s_1^*\vee s_2^*$.
\end{remark}
As we interpret the basic sentence types involved in the square of opposition by imprecise probability assessments on sequences of conditional events, we will introduce the following notion of acceptability, which serves as a semantic bridge between basic sentence types and imprecise assessments:
\begin{definition}\label{DEF:ACC}
A sentence $s:(\mathcal{F},\mathcal{I})$ is (resp., is not)   \emph{acceptable}  iff the assessment $\mathcal{I}$ on $\mathcal{F}$ is (resp., is not)  g-coherent, i.e. $\coh(\mathcal{I})$ is not (resp., is) empty.
\end{definition}
\begin{remark}
If $s_1\wedge s_2$ is acceptable, then $s_1$ is acceptable and $s_2$ is acceptable. However, the converse does not hold, indeed  $s_1:(E|H, \{1\})$ is acceptable and $s_2:(E|H),\{0\})$ is acceptable, but $s_1 \wedge s_2:(E|H, \emptyset)$ is not acceptable (as $\coh(\emptyset)=\emptyset$).  
\end{remark}
\begin{definition}\label{DEF:4REL}
Given two sentences $s_1:(\mathcal{F},\mathcal{I}_1)$ and $s_2:(\mathcal{F},\mathcal{I}_2)$, we say,  under coherence:
$s_1$ and $s_2$ are \emph{contraries}  iff the sentence $s_1\wedge s_2$ is not acceptable;\footnote{ 
Some definitions of contrariety
additionally require that ``$s_1$ and $s_2$ can both be acceptable''. 
For reasons stated in \cite{gilio16}, we omit  this additional
requirement.
Similarly, \emph{mutatis mutandis}, in our definition of subcontrariety.}
 $s_1$ and $s_2$ are \emph{subcontraries} iff $\no{s}_1\wedge  \no{s}_2$ is not acceptable;
$s_1$ and $s_2$ are \emph{contradictories} iff  $s_1$ and $s_2$ are both, contraries and subcontraries;
 $s_2$ is a \emph{subaltern} of  $s_1$ iff the sentence  $s_1 \wedge  \no{s}_2$ is not acceptable.
\end{definition}
\begin{remark}\label{REM:REL}
By Remark~\ref{REM:1}, we observe that two sentences $s_1$ and $s_2$ are contraries if and only if $\coh(\I_1\cap \I_2)=
\coh(\I_1)\cap\coh(\I_2)=\emptyset$. Moreover, two sentences $s_1$ and $s_2$ are subcontraries if and only if $\coh(\no{\I}_1\cap \no{\I}_2)=\coh(\no{\I}_1)\cap \coh(\no{\I}_2)=\emptyset$, that is (by De Morgan's law) if and only if $\coh(\I_1)\cup \coh(\I_2)=
\Pi$. Then, two sentences $s_1$ and $s_2$ are contradictories
if and only if $\coh(\I_1)\cap\coh(\I_2)=\emptyset$ and 
$\coh(\I_1)\cup \coh(\I_2)=
\Pi$, that is if and only if $s_2=\no{s}_1$ (and, of course, $s_1=\no{s}_2$).
Given two sentences $s_1,s_2$ we also observe that  $s_2$ is a subaltern of $s_1$ if and only if $\Pi\cap(\mathcal{I}_1 \cap \no{\mathcal{I}}_2)=\emptyset$, which also amounts to say that $\Pi \cap \mathcal{I}_1 \subseteq \Pi \cap \mathcal{I}_2$, that is if and only if  $\coh(\I_1)\subseteq \coh(\I_2)$. For instance, $s_1\vee s_2$ is a subaltern of  $s_1$ and also of $s_2$; similarly, $s_1$ is  a subaltern of $s_1\wedge s_2$, and $s_2$ is  a subaltern of $s_1\wedge s_2$.
Furthermore, if $s_1$ is not acceptable, that is $\coh( \mathcal{I}_1)=\emptyset$, then any sentence $s_2$ is a subaltern of $s_1$. For example, the sentence $s_1:(E|\no{E},\{1\})$ is not acceptable because $\Pi=\{0\}$ and then any sentence $s_2:(E|\no{E},\mathcal{I})$, where $\mathcal{I}\subseteq [0,1]$, is a subaltern of $s_1$.
\end{remark}
Based on the relations given in Definition~\ref{DEF:4REL} we define a square of opposition as follows.
\begin{definition}\label{DEF:SQ}
Let  $s_k:(\mathcal{F},\mathcal{I}_k)$, $k=1,2,3,4$,  be four sentences. We call the ordered quadruple $(s_1,s_2,s_3,s_4)$  a \emph{square of opposition} (under coherence),  iff the following relations among the four sentences hold:
\begin{enumerate}
\item[(a)] $s_1$ and $s_2$ are contraries, i.e., $\coh(\mathcal{I}_1)\cap \coh (\mathcal{I}_2)=\emptyset$;
\item[(b)] $s_3$ and $s_4$ are subcontraries, i.e., $\coh(\mathcal{I}_3)\cup \coh (\mathcal{I}_4)=\Pi$;
\item[(c)] $s_1$ and $s_4$ are contradictories, i.e., $\coh(\mathcal{I}_1)\cap \coh (\mathcal{I}_4)=\emptyset$ and   
${\coh(\mathcal{I}_1)\cup\coh (\mathcal{I}_4)=\Pi}$; 
\newline $s_2$ and $s_3$ are contradictories, i.e., $\coh(\mathcal{I}_2)\cap\coh (\mathcal{I}_3)=\emptyset$ and 
 ${\coh(\mathcal{I}_2)\cup\coh (\mathcal{I}_3)=\Pi}$; 
\item[(d)] $s_3$ is a subaltern of $s_1$, i.e., $\coh(\mathcal{I}_1)\subseteq\coh (\mathcal{I}_3)$;\newline $s_4$ is a subaltern of $s_2$, i.e., $\coh(\mathcal{I}_2)\subseteq\coh (\mathcal{I}_4)$.
\end{enumerate}
\end{definition}
Figure~\ref{fig:QsquareS} shows the square of opposition based on Definition~\ref{DEF:SQ}.
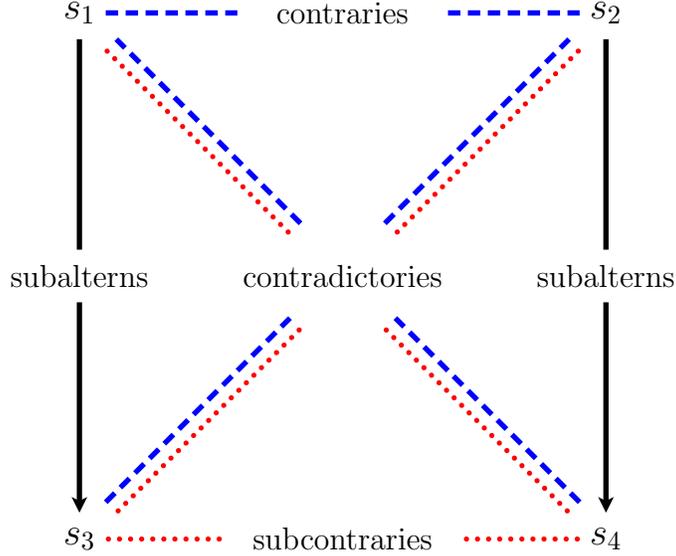
\begin{figure}[h]
\centering
\begin{center}
\ifx\JPicScale\undefined\def\JPicScale{.7}\fi
\psset{unit=\JPicScale mm}
\psset{linewidth=0.3,dotsep=1,hatchwidth=0.3,hatchsep=1.5,shadowsize=1,dimen=middle}
\psset{dotsize=0.7 2.5,dotscale=1 1,fillcolor=black}
\psset{arrowsize=1 2,arrowlength=1,arrowinset=0.25,tbarsize=0.7 5,bracketlength=0.15,rbracketlength=0.15}
\begin{pspicture}(0,-10)(100,120)
\rput(0,100){\large$s_1$}
\rput(0,0){\large$s_3$}
\rput(100,0){\large$s_4$}
\rput(100,100){\large$s_2$}
\rput(0,50){subalterns}
\rput(100,50){subalterns}
\rput(50,50){contradictories}
\rput(50,100){contraries}
\rput(50,0){subcontraries}
\psline[linewidth=2pt](0,95)(0,55)
\psline[linewidth=2pt]{->}(0,45)(0,05) 
\psline[linewidth=2pt](100,95)(100,55)  
\psline[linewidth=2pt]{->}(100,45)(100,05)

\psline[linestyle=dotted,linewidth=2pt,linecolor=red](7,5)(42,40) 
\psline[linestyle=dashed,linewidth=2pt,linecolor=blue](5,7)(40,42) 
\psline[linestyle=dotted,linewidth=2pt,linecolor=red](60,58)(95,93) 
\psline[linestyle=dashed,linewidth=2pt,linecolor=blue](58,60)(93,95) 
\psline[linestyle=dotted,linewidth=2pt,linecolor=red](93,5)(58,40)
\psline[linestyle=dashed,linewidth=2pt,linecolor=blue](95,7)(60,42)  
\psline[linestyle=dotted,linewidth=2pt,linecolor=red](40,58)(5,93) 
\psline[linestyle=dashed,linewidth=2pt,linecolor=blue](42,60)(7,95) 
\psline[linestyle=dotted,linewidth=2pt,linecolor=red](5,0)(27,0) 
\psline[linestyle=dotted,linewidth=2pt,linecolor=red](73,0)(95,0)
\psline[linestyle=dashed,linewidth=2pt,linecolor=blue](5,100)(30,100)
\psline[linestyle=dashed,linewidth=2pt,linecolor=blue](70,100)(95,100)
\end{pspicture}
\end{center}
\caption{Probabilistic square of opposition defined by the   quadruple $(s_1,s_2,s_3,s_4)$. The arrows indicate subalternation, dashed lines indicate contraries, and dotted lines indicate sub-contraries. Contradictories are indicated by combined dotted and dashed lines.}
\label{fig:QsquareS}
\end{figure}

\begin{remark}\label{REM:SQVERIFICATION}
Based on Definition~\ref{DEF:SQ}, we observe that in order to verify if a quadruple of sentences $(s_1,s_2,s_3,s_4)$, where $s_k:(\mathcal{F},\mathcal{I}_k)$ and $k=1,2,3,4$, is a square of opposition, it is necessary and sufficient to check that the quadruple $(s_1',s_2',s_3',s_4')$, where $s_k'=(\mathcal{F},\mathcal{I}_k')$, $\mathcal{I}_k'=\coh(\mathcal{I}_k)$, 
 is a square of opposition. Then, we say that two squares $(s_1,s_2,s_3,s_4)$ and $(s_1',s_2',s_3',s_4')$   \emph{coincide}  iff 
$\coh(\mathcal{I}_k)=\coh(\mathcal{I}_k')$ for each $k$. Moreover, based on Definition~\ref{DEF:SQ}, we observe that  $(s_1,s_2,s_3,s_4)$ is a square of opposition iff  $(s_2,s_1,s_4,s_3)$ is a square of opposition. 
\end{remark}
\begin{definition}
\def\Set{\mathfrak{S}}
An (ordered) \emph{tripartition} of a set $\Set$ is a triple   $(\mathcal{D}_1,\mathcal{D}_2,\mathcal{D}_3)$, where $\mathcal{D}_1$, $\mathcal{D}_2$,  and $\mathcal{D}_3$ are subsets of $\Set$, such that the following conditions are satisfied: (i) $\mathcal{D}_i\cap \mathcal{D}_j=\emptyset$, $i\neq j$ for all $i,j=1,2,3$; (ii);
$\mathcal{D}_1\cup \mathcal{D}_2 \cup \mathcal{D}_3=\Set$.
\end{definition}
\begin{theorem}\label{THM:SQDP}
Given any sequence of $n$ conditional events $\mathcal{F}$ and
a quadruple  $(s_1,s_2,s_3,s_4)$ of sentences,
with  $s_k:(\mathcal{F},\mathcal{I}_k)$, $k=1,2,3,4$. Define $\mathcal{D}_1=\coh(\mathcal{I}_1)$, $\mathcal{D}_2=\coh(\mathcal{I}_2)$, and $\mathcal{D}_3=\coh(\mathcal{I}_3)\cap\coh(\mathcal{I}_4)$.
Then, the quadruple $(s_1,s_2,s_3,s_4)$ is a square of opposition  if and only if
$(\mathcal{D}_1,\mathcal{D}_2,\mathcal{D}_3)$ is a tripartition  of (the non-empty set) $\Pi$ such that:  $\coh(\mathcal{I}_3)=\mathcal{D}_1\cup \mathcal{D}_3$, $\coh(\mathcal{I}_4)=\mathcal{D}_2\cup \mathcal{D}_3$.
\end{theorem}
\begin{proof}
$(\Rightarrow)$.
We assume that $\mathcal{D}_1=\coh(\mathcal{I}_1)$, $\mathcal{D}_2=\coh(\mathcal{I}_2)$, and $\mathcal{D}_3=\coh(\mathcal{I}_3)\cap\coh(\mathcal{I}_4)$. Of course, $\mathcal{D}_i\subseteq \Pi$, $i=1,2,3$. We now prove that: $(i)$ $\mathcal{D}_1\cap \mathcal{D}_2=\emptyset$; $(ii)$ $\mathcal{D}_3= \Pi\setminus (\mathcal{D}_1\cup \mathcal{D}_2)$.
$(i)$ From condition (a) in Definition~\ref{DEF:SQ}, as $s_1$ and $s_2$ are contraries, it follows that  $\mathcal{D}_1\cap\mathcal{D}_2=\emptyset$.
$(ii)$ We first prove that $\mathcal{D}_3\subseteq \Pi\setminus (\mathcal{D}_1\cup \mathcal{D}_2)$. This trivially follows when  $\mathcal{D}_3=\emptyset$. If $\mathcal{D}_3\neq\emptyset$, then let $x\in \mathcal{D}_3=\coh(\mathcal{I}_3)\cap\coh(\mathcal{I}_4)$.  As $x\in\coh(\mathcal{I}_3)$, from condition (c) in  Definition~\ref{DEF:SQ}, we obtain $x\notin\coh(\mathcal{I}_2)$. Likewise, as $x\in\coh(\mathcal{I}_4)$, from condition (c) in  Definition~\ref{DEF:SQ}, we obtain $x\notin\coh(\mathcal{I}_1)$. Then,  $x\in \Pi$ and $x\notin(\coh(\mathcal{I}_1) \cup\coh(\mathcal{I}_2))$, that is $x\in \Pi\setminus (\mathcal{D}_1\cup \mathcal{D}_2)$.  We now prove that $\Pi\setminus (\mathcal{D}_1\cup \mathcal{D}_2)\subseteq\mathcal{D}_3$.
This trivially follows when $\Pi\setminus (\mathcal{D}_1\cup \mathcal{D}_2)=\emptyset$. If $\Pi\setminus (\mathcal{D}_1\cup \mathcal{D}_2)\neq\emptyset$, let
  $x\in \Pi\setminus (\coh(\mathcal{I}_1)\cup\coh(\mathcal{I}_2))$.  As $x\in \Pi \setminus\coh(\mathcal{I}_1)$, from condition (c) in  Definition~\ref{DEF:SQ}, we obtain $x\in\coh(\mathcal{I}_4)$.  Likewise, as $x\in \Pi \setminus\coh(\mathcal{I}_2)$ from condition (c) in  Definition~\ref{DEF:SQ}, we obtain $x\in \coh(\mathcal{I}_3)$. Then, $x\in (\coh(\mathcal{I}_3)\cap \coh(\mathcal{I}_4))=\mathcal{D}_3$. Therefore $(\mathcal{D}_1,\mathcal{D}_2,\mathcal{D}_3)$ is a tripartition of $\Pi$. 
By our assumption, $\coh(\mathcal{I}_1)=\mathcal{D}_1$ and $\coh(\mathcal{I}_2)=\mathcal{D}_2$. We observe that $\coh(\mathcal{I}_3)\cap \mathcal{D}_3=\mathcal{D}_3$; moreover, from conditions (c) and (d), we obtain  $\coh(\mathcal{I}_3)\cap \mathcal{D}_2=\coh(\mathcal{I}_3)\cap \coh(\mathcal{I}_2)=\emptyset$ and $\coh(\mathcal{I}_3) \cap \mathcal{D}_1=\coh(\mathcal{I}_1)\cap \coh(\mathcal{I}_3)=\coh(\mathcal{I}_1)=\mathcal{D}_1$; then
$\coh(\mathcal{I}_3)=\coh(\mathcal{I}_3)\cap (\mathcal{D}_1\cup \mathcal{D}_2 \cup \mathcal{D}_3)=\mathcal{D}_1\cup \mathcal{D}_3$. Likewise, we observe that $\coh(\mathcal{I}_4)\cap \mathcal{D}_3=\mathcal{D}_3$; moreover, from conditions (c),(d) in Definition~\ref{DEF:SQ}, we obtain  $\mathcal{D}_1\cap \coh(\mathcal{I}_4)=\coh(\mathcal{I}_1)\cap \coh(\mathcal{I}_4)=\emptyset$ and 
 $\mathcal{D}_2\cap \coh(\mathcal{I}_4)=\coh(\mathcal{I}_2)\cap \coh(\mathcal{I}_4)=\coh(\mathcal{I}_2)=\mathcal{D}_2$; then
$\coh(\mathcal{I}_4)=\coh(\mathcal{I}_4)\cap (\mathcal{D}_1\cup \mathcal{D}_2 \cup \mathcal{D}_3)=\mathcal{D}_2\cup \mathcal{D}_3$.\\
$(\Leftarrow)$
Assume that $(\mathcal{D}_1,\mathcal{D}_2,\mathcal{D}_3)$, where $\mathcal{D}_1=\coh(\mathcal{I}_1)$, $\mathcal{D}_2=\coh(\mathcal{I}_2)$, $\mathcal{D}_3=\coh(\mathcal{I}_3)\cap \coh(\mathcal{I}_4)$,  is a tripartition of $\Pi$ such that  $\mathcal{D}_1\cup \mathcal{D}_3=\coh(\mathcal{I}_3)$ and $\mathcal{D}_2\cup \mathcal{D}_3=\coh(\mathcal{I}_4)$,  we prove that the quadruple $(s_1,s_2,s_3,s_4)$ satisfies conditions (a), (b), (c), and  (d) in Definition~\ref{DEF:SQ}.
We observe that $\coh(\mathcal{I}_1)\cap \coh(\mathcal{I}_2)=\mathcal{D}_1\cap \mathcal{D}_2=\emptyset$, which coincides with (a). 
Condition (b) is satisfied because  $\coh(\mathcal{I}_3)\cup \coh(\mathcal{I}_4)=\mathcal{D}_1\cup\mathcal{D}_3 \cup \mathcal{D}_2\cup \mathcal{D}_3=\Pi$. Moreover, $\coh(\mathcal{I}_1)\cap \coh(\mathcal{I}_4)=\mathcal{D}_1\cap (\mathcal{D}_2\cup \mathcal{D}_3)=\emptyset$ and
$\coh(\mathcal{I}_1)\cup \coh(\mathcal{I}_4)=\mathcal{D}_1\cup (\mathcal{D}_2\cup \mathcal{D}_3)=\Pi$; likewise, $\coh(\mathcal{I}_2)\cap \coh(\mathcal{I}_3)=\mathcal{D}_2\cap (\mathcal{D}_1\cup \mathcal{D}_3)=\emptyset$ and
$\coh(\mathcal{I}_2)\cup \coh(\mathcal{I}_3)=\mathcal{D}_2\cup (\mathcal{D}_1\cup \mathcal{D}_3)=\Pi$. Thus, the conditions in  (c) are satisfied. Finally, $\coh(\mathcal{I}_1)=\mathcal{D}_1\subseteq \mathcal{D}_1\cup\mathcal{D}_3=\coh(\mathcal{I}_3)$ and $\coh(\mathcal{I}_2)=\mathcal{D}_2\subseteq \mathcal{D}_2\cup\mathcal{D}_3=\coh(\mathcal{I}_4)$ which satisfy conditions in (d). 
\end{proof}
A method to  construct a square of opposition by starting from a tripartition of $\Pi$ is given in the following result (see also \cite{DuboisPrade2012}).
\begin{corollary}\label{COR:TRIP} Given any sequence of $n$ conditional events $\mathcal{F}$ and a tripartition $(\mathcal{D}_1,\mathcal{D}_2,\mathcal{D}_3)$ of $\Pi$, then
the quadruple $(s_1,s_2,s_3,s_4)$, with  $s_k:(\mathcal{F},\mathcal{I}_k)$, $k=1,2,3,4$ and $\coh(\mathcal{I}_1)=\mathcal{D}_1$, $\coh(\mathcal{I}_2)=\mathcal{D}_2$,  $\coh(\mathcal{I}_3)=\mathcal{D}_1\cup \mathcal{D}_3$, $\coh(\mathcal{I}_4)=\mathcal{D}_2\cup \mathcal{D}_3$ is a square of opposition.
\end{corollary}
\begin{proof}
The proof immediately follows by observing $\coh(\mathcal{I}_3)\cap \coh(\mathcal{I}_4)=\mathcal{D}_3$ and by the ($\Leftarrow$) side proof of Theorem~\ref{THM:SQDP}. 
\end{proof}
The following result allows to construct a square of opposition by starting from a tripartition of the whole set $[0,1]^n$:
\begin{corollary}\label{COR:NTRIP}
Given a tripartition
$(\mathcal{B}_1,\mathcal{B}_2,\mathcal{B}_3)$  of $[0,1]^n$, let  $\mathcal{I}_1=\mathcal{B}_1$,  $\mathcal{I}_2=\mathcal{B}_2$, $\mathcal{I}_3=\mathcal{B}_1\cup \mathcal{B}_3$, and  $\mathcal{I}_4=\mathcal{B}_2\cup \mathcal{B}_3$. For any sequence of $n$ conditional events $\mathcal{F}$, the quadruple $(s_1,s_2,s_3,s_4)$, where  $s_k:(\mathcal{F},\mathcal{I}_k)$, $k=1,2,3,4$, is a square of opposition.
\end{corollary}
\begin{proof}
Let $\mathcal{F}$ be any sequence of $n$ conditional events and $\Pi$ be the associated set of all coherent precise assessments. We set $\mathcal{D}_i=\coh(\mathcal{B}_i)$, $i=1,2,3$. Of course,
$(\coh(\mathcal{B}_1),\coh(\mathcal{B}_2),\coh(\mathcal{B}_3))$ is a tripartition of $\Pi$. Moreover, $\coh(\mathcal{I}_1)=\mathcal{D}_1$, $\coh(\mathcal{I}_2)=\mathcal{D}_2$,  $\coh(\mathcal{I}_3)=\mathcal{D}_1\cup \mathcal{D}_3$, $\coh(\mathcal{I}_4)=\mathcal{D}_2\cup \mathcal{D}_3$. Then, by applying  Corollary~\ref{COR:TRIP} we obtain that $(s_1,s_2,s_3,s_4)$ is a square of opposition. 
\end{proof}
Traditionally the square of opposition can be  constructed based on the fragmented square of opposition which requires only  the contrariety and contradiction relations (which goes back to  Aristotle's \emph{De Interpretatione} 6--7, 17b.17--26, see \cite[Section\ 2]{sep-square}).
This result also holds in our framework:
\begin{theorem}\label{THM:SquareAC}
The quadruple  $(s_1,s_2,s_3,s_4)$ of sentences, 
with  $s_k:(\mathcal{F},\mathcal{I}_k)$, $k=1,2,3,4$, is a square of opposition  iff  relations (a) and (c) in Definition~\ref{DEF:SQ} are satisfied.  
\end{theorem}
\begin{proof}
$(\Rightarrow)$ It follows directly from Definition \ref{DEF:SQ}. $(\Leftarrow)$  We prove that  (d) and (b) in Definition \ref{DEF:SQ} follow from  (a) and (c). If $\coh(\mathcal{I}_1)=\emptyset$, then of course  $\coh(\mathcal{I}_1)\subseteq \coh(\mathcal{I}_3)$. If $\coh(\mathcal{I}_1)\neq\emptyset$, 
let  $x\in \coh(\mathcal{I}_1)\subseteq \Pi$, from  (a)  it follows that $x\notin \coh(\mathcal{I}_2)$, and   since  (c) requires $\coh(\mathcal{I}_2)\cup \coh(\mathcal{I}_3)=\Pi$, we obtain  $x\in \coh(\mathcal{I}_3)$. Thus, $\coh(\mathcal{I}_1)\subseteq \coh(\mathcal{I}_3)$; likewise,  $\coh(\mathcal{I}_2)\subseteq \coh(\mathcal{I}_4)$. Therefore,  (d) is satisfied. Now we prove that (b) is satisfied, i.e. $\coh(\mathcal{I}_3)\cup \coh(\mathcal{I}_4)=\Pi$. Of course, $\coh(\mathcal{I}_3)\cup \coh(\mathcal{I}_4)\subseteq \Pi$.   Let $x\in \Pi$. If  $x\notin \coh(\mathcal{I}_3)$, then,  $x\in \coh(\mathcal{I}_2)$ from (c). Moreover, from (d), $x\in \coh(\mathcal{I}_4)$. Then, $\Pi\subseteq \coh(\mathcal{I}_3)\cup \coh(\mathcal{I}_4)$. Therefore, (b) is satisfied.  
\end{proof}

\begin{corollary}\label{COR:EASYSQUARE}
The quadruple  $(s_1,s_2,s_3,s_4)$ of sentences, 
with  $s_k:(\mathcal{F},\mathcal{I}_k)$, $k=1,2,3,4$, is a square of opposition  if and only if   $(s_1,s_2,s_3,s_4)=(s_1,s_2,\no{s}_2,\no{s}_1)$
with
$s_1$ and $s_2$ being contraries.
\end{corollary}
\begin{proof}
Of course, if $(s_1,s_2,s_3,s_4)$ is a square of opposition, then $s_1$ and $s_2$ are contraries. 
Moreover, $s_1$ and $s_4$ are contradictories, that is: $\coh(\mathcal{I}_1)\cap \coh(\mathcal{I}_4)=\emptyset$ and
$\coh(\mathcal{I}_1)\cup \coh(\mathcal{I}_4)=\Pi$. Therefore, $\Pi \setminus \coh(\mathcal{I}_4)=\coh(\mathcal{I}_1)$, which amounts to $s_4=\no{s}_1$. Similary, as $s_2$ and $s_3$ are contradictories, it holds that $s_3=\no{s}_2$.
Conversely, assume that $s_1$ and $s_2$ are contraries.
By instantiating Theorem~\ref{THM:SquareAC} with $s_3=\no{s}_2$ and with $s_4=\no{s}_1$, it follows that the quadruple $(s_1,s_2,\no{s}_2,\no{s}_1)$ is a square of opposition.
\end{proof}

\section{Square of Opposition  and Generalized Quantifiers}\label{SEC:GSQUARE}
Let $\mathcal{F}$ be a conditional event $P|S$ (where $S\neq \bot$)  and  $(\mathcal{B}_1(x),\mathcal{B}_2(x),\mathcal{B}_3(x))$  be a  tripartition of $[0,1]$, where $\mathcal{B}_1(x)=[x,1]$, $\mathcal{B}_2(x)=[0,1-x]$, $\mathcal{B}_3(x)=]1-x,x[$ and $x\in]\frac12,1]$ (see Figure~\ref{FIG:TRIPARTITIONB}). 

\begin{figure}
\centering
\begin{pspicture}(0,7)(9.5,9)
\definecolor{colour0}{rgb}{0.7,0.7,0.7}
\psline[linecolor=colour0, linewidth=0.1,tbarsize=0.5306cm ,bracketlength=0.15]{[-]}(0,7.5)(4,7.5)
\psline[linecolor=colour0, linewidth=0.1,tbarsize=0.5306cm ,bracketlength=0.15]{-}(4,7.5)(6,7.5)
\psline[linecolor=colour0, linewidth=0.1,tbarsize=0.5306cm ,bracketlength=0.15]{[-]}(6,7.5)(10,7.5)
\psline[linecolor=black, linewidth=0.05]{|-|}(0,7.5)(10,7.5)
\rput(8,8){\large $\mathcal{B}_1(x)$}
\rput(2,8){\large $\mathcal{B}_2(x)$}
\rput(5,8){\large $\mathcal{B}_3(x)$}
\rput(0,7){\large $0$}
\rput(10,7){\large $1$}
\rput(4,7){\large $1-x$}
\rput(6,7){\large $x$}
\end{pspicture}

\caption{\label{FIG:TRIPARTITIONB} Example of a tripartition $(\mathcal{B}_1(x), \mathcal{B}_2(x), \mathcal{B}_3(x))$ of $[0,1]$, with $x\in]\frac12,1]$.}
\end{figure}
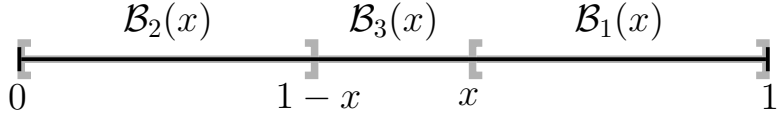

Consider the quadruple of sentences $(A(x),E(x),I(x),O(x))$, with $A(x):(P|S,\mathcal{I}_{A(x)})$, $E(x):(P|S,\mathcal{I}_{E(x)})$, $I(x):(P|S,\mathcal{I}_{I(x)})$, $O(x):(P|S,\mathcal{I}_{O(x)})$, where 
$\mathcal{I}_{A(x)}=\mathcal{B}_1(x)=[x,1]$,
$\mathcal{I}_{E(x)}=\mathcal{B}_2(x)=[0,1-x]$,
$\mathcal{I}_{I(x)}=\mathcal{B}_1(x) \cup \mathcal{B}_3(x)=]1-x,1]$, and $\mathcal{I}_{O(x)}=\mathcal{B}_2(x)\cup \mathcal{B}_3(x)=[0,x[$.
By applying  Corollary \ref{COR:NTRIP} with $(s_1,s_2,s_3,s_4)=(A(x),E(x),I(x),O(x))$, it follows that  $(A(x),E(x),I(x),O(x))$ is a square of opposition for any $x\in]\frac12,1]$ (see Figure \ref{fig:Qsquare}).
\begin{figure}[h]
\centering
\begin{center}
\ifx\JPicScale\undefined\def\JPicScale{.7}\fi
\psset{unit=\JPicScale mm}
\psset{linewidth=0.3,dotsep=1,hatchwidth=0.3,hatchsep=1.5,shadowsize=1,dimen=middle}
\psset{dotsize=0.7 2.5,dotscale=1 1,fillcolor=black}
\psset{arrowsize=1 2,arrowlength=1,arrowinset=0.25,tbarsize=0.7 5,bracketlength=0.15,rbracketlength=0.15}
\begin{pspicture}(0,-10)(100,120)
\rput(-2,100){$A(x)$}
\rput(0,106){$p(P|S)\geq x$}
\rput(-2,0){$I(x)$}
\rput(0,-6){$p(P|S)>1-x$}
\rput(102,0){$O(x)$}
\rput(100,-6){$p(P|S)<x$}
\rput(102,100){E(x)}
\rput(100,106){$p(P|S)\leq 1-x$}
\rput(0,50){subalterns}
\rput(100,50){subalterns}
\rput(50,50){contradictories}
\rput(50,100){contraries}
\rput(50,0){subcontraries}
\psline[linewidth=2pt](0,95)(0,55)
\psline[linewidth=2pt]{->}(0,45)(0,05) 
\psline[linewidth=2pt](100,95)(100,55)  
\psline[linewidth=2pt]{->}(100,45)(100,05)

\psline[linestyle=dotted,linewidth=2pt,linecolor=red](7,5)(42,40) 
\psline[linestyle=dashed,linewidth=2pt,linecolor=blue](5,7)(40,42) 
\psline[linestyle=dotted,linewidth=2pt,linecolor=red](60,58)(95,93) 
\psline[linestyle=dashed,linewidth=2pt,linecolor=blue](58,60)(93,95) 
\psline[linestyle=dotted,linewidth=2pt,linecolor=red](93,5)(58,40)
\psline[linestyle=dashed,linewidth=2pt,linecolor=blue](95,7)(60,42)  
\psline[linestyle=dotted,linewidth=2pt,linecolor=red](40,58)(5,93) 
\psline[linestyle=dashed,linewidth=2pt,linecolor=blue](42,60)(7,95) 
\psline[linestyle=dotted,linewidth=2pt,linecolor=red](5,0)(27,0) 
\psline[linestyle=dotted,linewidth=2pt,linecolor=red](73,0)(95,0)
\psline[linestyle=dashed,linewidth=2pt,linecolor=blue](5,100)(30,100)
\psline[linestyle=dashed,linewidth=2pt,linecolor=blue](70,100)(95,100)
\end{pspicture}
\end{center}
\caption{Probabilistic square of opposition $\mathbf{S}(x)$ defined on the four sentence types $(A(x),E(x),I(x),O(x))$  with the threshold $x\in]\frac12,1]$ (see also Table~\ref{Table:GS}). It provides a new interpretation of the traditional square of opposition (see, e.g., \cite{sep-square}), where the corners are labeled by ``Every $S$ is $P$'' (A), ``No $S$ is $P$'' (E), ``Some $S$ is $P$'' (I), and ``Some $S$ is not $P$'' (O).}
\label{fig:Qsquare}
\end{figure}
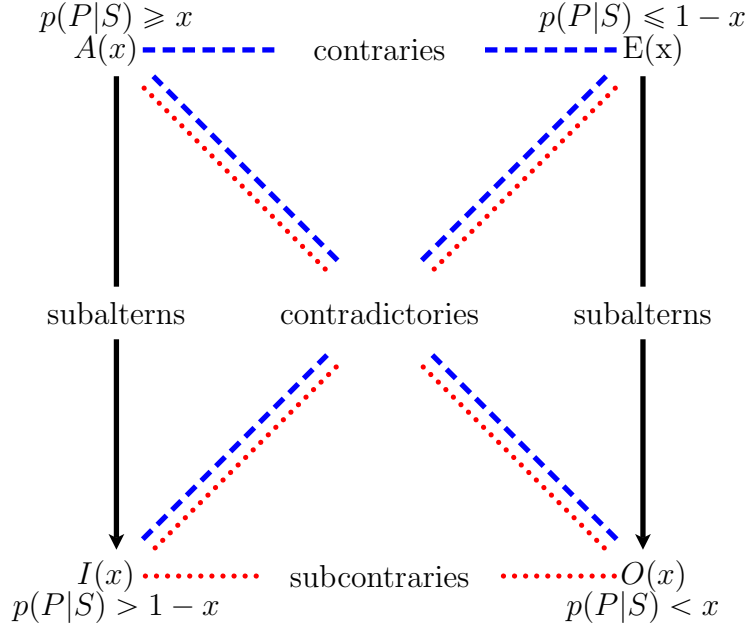
We recall that in presence of  some logical relations between $P$ and $S$ the set $\Pi$ could be a strict subset of $[0,1]$. In particular, we have  the following three cases (see, \cite{gilio13ijar,gilio13ins}):
(i) if $P\wedge S\neq \bot$ and $P\wedge S\neq S$, then $\Pi=[0,1]$; (ii) if  $P\wedge S=S$, then $\Pi=\{1\}$;
(iii) if $P\wedge S=\bot$, then $\Pi=\{0\}$.
The quadruple $(A(x),E(x),I(x),O(x))$, with the threshold $\frac{1}{2}<x\leq 1$, is a square of opposition in each of the three cases.
In particular we obtain: 
case (i)
$\coh(\mathcal{I}_{A(x)})=\mathcal{I}_{A(x)}$, $\coh(\mathcal{I}_{E(x)})=\mathcal{I}_{E(x)}$,$\coh(\mathcal{I}_{I(x)})=\mathcal{I}_{I(x)}$, and $\coh(\mathcal{I}_{O(x)})=\mathcal{I}_{O(x)}$;
case (ii):
$\coh(\mathcal{I}_{A(x)})=\{1\}$, $\coh(\mathcal{I}_{E(x)})=\emptyset$,$\coh(\mathcal{I}_{I(x)})=\{1\}$, and $\coh(\mathcal{I}_{O(x)})=\emptyset$;
case (iii): 
$\coh(\mathcal{I}_{A(x)})=\emptyset$, $\coh(\mathcal{I}_{E(x)})=\{1\}$,$\coh(\mathcal{I}_{I(x)})=\emptyset$, and $\coh(\mathcal{I}_{O(x)})=\{1\}$.
We note that in cases (ii) and (iii) we obtain degenerated  squares each, where---apart from the contradictory relations---all relations are strengthened
. Specifically, 
both  contrary and the subcontrary become contradictory relations. Moreover, both subalternation relations become symmetric. As by coherence $p(P|S)+p(\no{P}|S)=1$,  a sentence $s:(P|S,\mathcal{I})$  is equivalent to the sentence $s':(\no{P}|S,\no{\mathcal{I}})$, where $\no{\mathcal{I}}=[0,1]\setminus \mathcal{I}$. Table~\ref{Table:GS}
\begin{table}[!h]
\centering
\begin{tabular}{llll}\hline \hline
\multicolumn{2}{l}{Sentence} & Probability constraints & Assessment on $P|S$  \\
        \hline
$A(x):$& $(Q_{\geq x}$ $S$ are $P$) &$p(P|S)\geq x$ &    ${\mathcal{I}}_{A(x)}=[x,1]$\\
$E(x):$& $(Q_{\geq x}$ $S$ are not $P$) & $p(\no{P}|S)\geq x$ &   ${\mathcal{I}}_{E(x)}=[0,1-x]$ \\
$I(x):$& ($Q_{>1-x}$ $S$ are $P$) & $p(P|S)>1-x$ & ${\mathcal{I}}_{I(x)}=]1-x,1]$       \\
$O(x):$& ($Q_{>1-x}$ $S$ are not $P$) & $p(\no{P}|S)>1-x$ &${\mathcal{I}}_{O(x)}=[0,x[$\\ 
\hline        
$A(1):$& (Every $S$ is $P$) &$p(P|S)=1$ &    $\mathcal{I}_A=\{1\}$\\
$E(1):$& (No $S$ is  $P$) & $p(\no{P}|S)=1$ &   $\mathcal{I}_E=\{0\}$ \\
$I(1):$& (Some $S$ is $P$) & $p(P|S)>0$ & $\mathcal{I}_I=]0,1]$       \\
$O(1):$& (Some $S$ is not $P$) & $p(\no{P}|S)>0$ &$\mathcal{I}_O=[0,1[$        \\
 \hline \hline
 \end{tabular}
\caption{Probabilistic interpretation of the sentence types $A$, $E$, $I$, and $O$ involving generalized quantifiers $Q$ defined by a threshold $x$ (with $x\in]\frac12,1]$) on the subject $S$ and predicate $P$ and the respective imprecise probabilistic assessments ${\mathcal{I}}_{A(x)}$, $\mathcal{I}_{E(x)}$, $\mathcal{I}_{I(x)}$, and $\mathcal{I}_{O(x)}$  on the conditional event $P|S$ (above). When $x=1$, we obtain our probabilistic interpretation of the  traditional sentence types $A$, $E$, $I$, and $O$  (below).}
\label{Table:GS}
\end{table}
 presents  generalization of  basic sentence types $A(x)$, $E(x)$, $I(x)$, and $O(x)$ involving generalized quantifiers $Q$. The generalized quantifiers are defined on a threshold $x>\frac12$. The  value of the threshold may be context dependent and provides lots of flexibility for modeling various instances of generalized quantifiers (like ``most'', ``almost all''). 
 
 Given two thresholds $x_1$ and $x_2$, with  $\frac{1}{2}<x_2<x_1\leq 1$, we analyze the relations  among the  same sentence types in the two squares of opposition $\mathbf{S}(x_1)$ and $\mathbf{S}(x_2)$, with $\mathbf{S}(x_i)=(A(x_i),E(x_i),I(x_i),O(x_i))$, $i=1,2$. 
 It can be easily proved that:
 $A(x_2)$ is a subaltern of   $A(x_1)$,
 $E(x_2)$ is a subaltern of   $E(x_1)$,
 $I(x_1)$ is a subaltern of   $I(x_2)$, and 
 $O(x_1)$ is a subaltern of   $O(x_2)$. 
 In the extreme case $x=1$ we obtain the probabilistic interpretation under coherence of the basic sentence types involved in the traditional square of opposition $(A,E,I,O)$ (see \cite{gilio15ecsqaru,gilio16} for the \emph{default square of opposition}, involving defaults and negated defaults). 
 
 In agreement with De Morgan (as pointed out by \cite{DuboisPrade2012}) by the quadruple $(a,e,i,o)$ we denotes the square of opposition obtained from $(A,E,I,O)$ when the events $P$ and $S$ are replaced by $\no{P}$ and $\no{S}$, respectively. Specifically, $a:(\no{P}|\no{S},\{1\})$, $e:(\no{P}|\no{S},\{0\})$, $i:(\no{P}|\no{S},]0,1])$, and $o:(\no{P}|\no{S},[0,1[)$. 

In the general case when $P$ and $S$ are  logically independent it can be proved that the set of all coherent assessments on $(P|S,\no{P}|\no{S})$ is the square $[0,1]^2$ (see e.g. \cite{coletti02}; see also \cite[Proposition 1]{coletti12} \cite[Theorem 4]{coleti14FSS}). Thus, in the general case there are no relations  between any two sentences $s_1$ and $s_2$, where $s_1\in\{A,E,I,O\}$ and $s_2\in\{a,e,i,o\}$. Therefore,  the two squares $(A,E,I,O)$ and $(a,e,i,o)$ do not form a cube of opposition (with these two squares as opposite facing sides).

\section{Hexagon of Opposition}
 \label{SEC:HEX}
Compared to the millennia long history of investigations on the square of opposition, the hexagon of opposition was discovered fairly recently, namely in the  1950ies. The hexagon generalizes the square by adding the disjunction of the top vertices of the square to build a new vertex at the top and by adding the conjunction of the bottom vertices of the square to build a new vertex at the bottom.
According to B\'{e}ziau (\cite{beziau2012}), the hexagon of opposition was introduced by the French priest and logician Augustin Sesmat (\cite{sesmat1951}) and by the philosopher Robert Blanch\'e (\cite{blanche1952}), who worked out the full structure of the hexagon of opposition (for his main work on the hexagon of opposition see \cite{blanche1966}). Jaspers and Seuren (\cite{jaspers16}) trace the history of the hexagon  back also to the American philosopher Paul Jacoby (\cite{jacoby50}, see also \cite{DuboisPrade2012}).  
In this section we will use the tools developed in Section~\ref{SEC:SQUARE}, to  construct a \emph{hexagon of opposition} by starting from a square of opposition. More precisely,  given a traditional square of opposition  $(A,E,I,O)$, by setting $U=A\lor E$, $Y=I\land O$, the  tuple $(A,E,I,O,U,Y)$ defines a hexagon of opposition. Accordingly,  we define the (probabilistic) hexagon of opposition in our approach as follows: 
\begin{definition}[Hexagon of opposition]\label{DEF:HEX}
Let  $s_k:(\mathcal{F},\mathcal{I}_k)$, $k=1,2,3,4,5,6$,  be six sentences. We call the ordered tuple $(s_1,s_2,s_3,s_4,s_5,s_6)$  a \emph{hexagon of opposition} (under coherence),  if and only if  the following relations among the six sentences hold:
\begin{enumerate}
\item[(i)] $(s_1,s_2,s_3,s_4)$ is a square of opposition;
\item[(ii)] $s_5=s_1 \vee s_2$;
\item[(iii)] $s_6=s_3 \wedge s_4$.
\end{enumerate}
\end{definition}

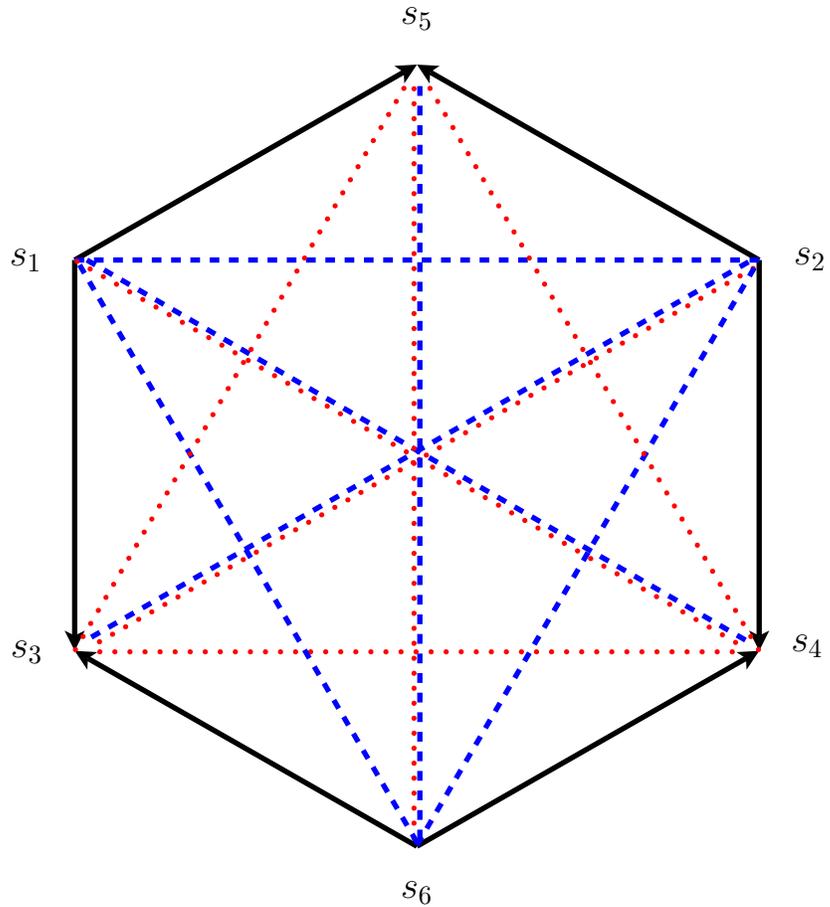
\begin{figure}
\ifx\JPicScale\undefined\def\JPicScale{1.3}\fi
\psset{unit=\JPicScale mm}
\psset{linewidth=0.3,dotsep=1,hatchwidth=0.3,hatchsep=1.5,shadowsize=1,dimen=middle}
\psset{dotsize=0.7 2.5,dotscale=1 1,fillcolor=black}
\psset{arrowsize=1 2,arrowlength=1,arrowinset=0.25,tbarsize=0.7 5,bracketlength=0.15,rbracketlength=0.15}
\begin{pspicture}(0,0)(100.2,89.7)
\psbezier[linewidth=2pt]{<-}(25,25)(25,25)(25,58.12)(25,65)
\psline[linewidth=2pt]{<-}(95,25)(95,65)
\psline[linewidth=2pt]{<-}(25,25)(60,5)
\psline[linewidth=2pt]{<-}(60,85)(95,65)
\psline[linewidth=2pt]{->}(25,65)(60,85)
\psline[linewidth=2pt]{->}(60,5)(95,25)
\rput(100.2,65.1){\large $s_2$}
\rput(60,89.7){\large $s_5$}
\rput(20,65){\large $s_1$}
\rput(20.1,24.9){\large $s_3$}
\rput(60,0.3){\large $s_6$}
\rput(99.9,25.5){\large $s_4$}
\psline[linewidth=2pt,linestyle=dashed,dash=1 1,linecolor=blue](25,65)(95,65)
\psline[linewidth=2pt,linestyle=dotted,linecolor=red](27.3,24.9)(92.4,24.9)
\psline[linewidth=2pt,linestyle=dashed,dash=1 1,linecolor=blue](26.7,26.4)(93.9,64.8)
\psline[linewidth=2pt,linestyle=dotted,linecolor=red](27.3,25.8)(94.5,64.2)
\psline[linewidth=2pt,linestyle=dashed,dash=1 1,linecolor=blue](26.25,65)(93.6,26.1)
\psline[linewidth=2pt,linestyle=dotted,linecolor=red](25,65)(92.7,25.8)
\psline[linewidth=2pt,linestyle=dotted,linecolor=red](59.7,82.8)(59.7,5.7)
\psline[linewidth=2pt,linestyle=dashed,dash=1 1,linecolor=blue](60.3,82.8)(60.3,5.7)
\psline[linewidth=2pt,linestyle=dashed,dash=1 1,linecolor=blue](25.5,64.5)(60,5.4)
\psline[linewidth=2pt,linestyle=dashed,dash=1 1,linecolor=blue](94.5,64.2)(60,5.1)
\psline[linewidth=2pt,linestyle=dotted,linecolor=red](58.8,82.8)(25,25)
\psline[linewidth=2pt,linestyle=dotted,linecolor=red](61.2,82.8)(95,25)
\end{pspicture}
\caption{\label{FIG:HexS}Probabilistic hexagon of opposition defined on the six sentence types $(s_1,s_2,s_3,s_4,s_5,s_6)$, where $(s_1,s_2,s_3,s_4)$ is a square of opposition, $s_5=s_1 \vee s_2$, and $s_6=s_3 \land s_4$
 (see Definition~\ref{DEF:HEX}). For the meaning of the lines see Figure~\ref{fig:QsquareS}.}
\end{figure}
Figure~\ref{FIG:HexS} shows the probabilistic hexagon of opposition as given by Definition~\ref{DEF:HEX}. 

\begin{theorem}\label{THM:HEX}
Let  $s_k:(\mathcal{F},\mathcal{I}_k)$, $k=1,2,3,4,5,6$,  be six sentences. The tuple $(s_1,s_2,s_3,s_4,s_5,s_6)$  is a hexagon of opposition,  if and only if
$(s_1,s_2,s_3,s_4,s_5,s_6)=(s_1,s_2,\no{s}_2,\no{s}_1,s_1 \vee s_2,\no{s}_1\wedge\no{s}_2)$, with  $s_1$ and $s_2$ being contraries.
\end{theorem}
\begin{proof}
 $(\Rightarrow)$. Let  $(s_1,s_2,s_3,s_4,s_5,s_6)$ be a hexagon of opposition. Then, as $(s_1,s_2,s_3,s_4)$ is a square of opposition,   $s_1$ and $s_2$ are contraries. Moreover, by Corollary~\ref{COR:EASYSQUARE}, it follows that $(s_1,s_2,s_3,s_4)=(s_1,s_2,\no{s}_2,\no{s}_1)$. Then, by Definition~\ref{DEF:HEX}, $s_5=s_1 \vee s_2$ and $s_6=s_3 \wedge s_4=\no{s}_1\wedge\no{s}_2$. Therefore, $(s_1,s_2,s_3,s_4,s_5,s_6)=(s_1,s_2,\no{s}_2,\no{s}_1,s_1 \vee 
 s_2,\no{s}_1\wedge\no{s}_2)$. \\
 $(\Leftarrow)$. 
 Let $(s_1,s_2,s_3,s_4,s_5,s_6)=(s_1,s_2,\no{s}_2,\no{s}_1,s_1 \vee 
 s_2,\no{s}_1\wedge\no{s}_2)$, with $s_1$ and $s_2$ being contraries. 
 From  Corollary~\ref{COR:EASYSQUARE}, it follows that $(s_1,s_2,s_3,s_4)$ is a square of opposition. Then, by relations $(ii)$ and $(iii)$ in  Definition~\ref{DEF:HEX}, it follows that $(s_1,s_2,s_3,s_4,s_5,s_6)$ is a hexagon of opposition.

\end{proof}
\begin{remark}
Assume that $s_1$ and $s_2$ are contraries. Then, 
by Corollary \ref{COR:EASYSQUARE}, the 
quadruple $(s_1,s_2,\no{s}_2,\no{s}_1)$ is a square of opposition, and  by Definition~\ref{DEF:HEX}, the
tuple $(s_1,s_2,\no{s}_2,\no{s}_1,s_1 \vee s_2,\no{s}_1\wedge\no{s}_2)$  is a hexagon of opposition.
\end{remark}
We now consider  relations among a tripartition of the set of all coherent assessments $\Pi$ and a hexagon of opposition.
\begin{remark}
Given a hexagon  of opposition $(s_1,s_2,s_3,s_4,s_5,s_6)$,
we observe that the sentence $s_6=s_3 \wedge s_4$ represents  the pair  $(\mathcal{F},\I_{6})$, where $\I_6=\mathcal{I}_3\cap \mathcal{I}_4$. Moreover,  by Remark~\ref{REM:1},  $\pi(\I_6)=\pi(\mathcal{I}_3\cap \mathcal{I}_4)=
\pi(\mathcal{I}_3)\cap \pi(\mathcal{I}_4)$.
Therefore, based on Theorem~\ref{THM:SQDP}, the triple
$(\mathcal{D}_1,\mathcal{D}_2,\mathcal{D}_3)$, where 
$\mathcal{D}_1=\pi(\I_{1})$, $\mathcal{D}_2=\pi(\I_{2})$, and $\mathcal{D}_3=\pi(\I_{6})$, is a tripartition of $\Pi$. Conversely, based on Corollary \ref{COR:TRIP}, given a tripartition $(\mathcal{D}_1,\mathcal{D}_2,\mathcal{D}_3)$ of $\Pi$, the sequence $(s_1,s_2,s_3,s_4,s_5,s_6)$ where 
 ${s_k:(\mathcal{F},\mathcal{I}_k)}$, $k=1,\ldots, 6$, with  $\coh(\mathcal{I}_1)=\mathcal{D}_1$, $\coh(\mathcal{I}_2)=\mathcal{D}_2$,  $\coh(\mathcal{I}_3)=\mathcal{D}_1\cup \mathcal{D}_3$, $\coh(\mathcal{I}_4)=\mathcal{D}_2\cup \mathcal{D}_3$,
 $\coh(\mathcal{I}_5)=\mathcal{D}_1\cup \mathcal{D}_2$, and  $\coh(\mathcal{I}_6)=\mathcal{D}_3$, is a hexagon of opposition (see also  \cite{Ciucci2015,DuboisPrade2012,Dubois2015}).
 \end{remark}
 Next, we consider  relations among a tripartition of $[0,1]^n$ and a hexagon of opposition.
 \begin{remark}\label{REM:HEXTRIP}
 Based on  Corollary~\ref{COR:NTRIP}, we can also  construct a hexagon of opposition by starting from a tripartition of the whole set $[0,1]^n$. Specifically,   given a tripartition $(\mathcal{B}_1,\mathcal{B}_2,\mathcal{B}_3)$  of $[0,1]^n$, let  $\mathcal{I}_1=\mathcal{B}_1$,  $\mathcal{I}_2=\mathcal{B}_2$, $\mathcal{I}_3=\mathcal{B}_1\cup \mathcal{B}_3$,  $\mathcal{I}_4=\mathcal{B}_2\cup \mathcal{B}_3$,
$\mathcal{I}_5=\mathcal{B}_1\cup \mathcal{B}_2$, and
$\mathcal{I}_6=\mathcal{B}_3$.
 For any sequence of $n$ conditional events $\mathcal{F}$, the tuple $(s_1,s_2,s_3,s_4,s_5,s_6)$, where  $s_k:(\mathcal{F},\mathcal{I}_k)$, $k=1,\ldots,6$, is a hexagon  of opposition. 
\end{remark}

\begin{theorem}\label{THM:HEXREL}
Given a hexagon of opposition $(s_1,s_2,s_3,s_4,s_5,s_6)$, by Definition \ref{DEF:HEX} all relations among the basic sentence types in  the square $(s_1,s_2,s_3,s_4)$ hold.  Moreover, by Theorem~\ref{THM:HEX} (and also by Remark~\ref{REM:REL}),  the following relations hold:
\begin{enumerate}[(i)]
    \item $s_1$ and $s_6$ 
     are contraries
    (since $s_6=\no{s}_2\wedge \no{s}_1$ and  $\coh(\I_1\cap     \no{\I}_2\cap \no{\I}_1)=\emptyset$); \label{ENU:1}
    \item $s_2$ and $s_6$ are contraries (since $s_6=\no{s}_2\wedge \no{s}_1$ and  $\coh(\I_2\cap     \no{\I}_2\cap \no{\I}_1)=\emptyset$);
	\item $s_3$ is a subaltern of $s_6$ (since $s_6=s_3\wedge s_4$);
	\item $s_4$ is a subaltern of $s_6$ (since $s_6=s_3\wedge s_4$);
	\item $s_5$ is a subaltern of $s_1$
	(since $s_5=s_1\vee s_2$);
	\item $s_5$ is a subaltern of $s_2$ (since $s_5=s_1\vee s_2$);
	\item $s_5$ and $s_3$ are subcontraries (as $s_5=s_1\vee s_2$ and $s_3=\no{s}_2$, hence 
	$\coh(\I_1\cup \I_2) \cup \no{\I}_2)=\Pi$);
	\item $s_5$ and $s_4$ are subcontraries
(as $s_5=s_1\vee s_2$ and $s_4=\no{s}_1$, hence 
	$\coh(\I_1\cup \I_2) \cup \no{\I}_1)=\Pi$);
	\item $s_5$ and $s_6$ are contradictories (as $s_5=s_1\vee s_2$, $s_6=s_3\wedge s_4=\no{s}_2\wedge \no{s}_1$,  hence 
	$\coh((\I_1\cup \I_2)\cap (\no{\I}_1\cap \no{\I}_2))  =\emptyset$) and 
	$\coh((\I_1\cup \I_2)\cup (\no{\I}_1\cap \no{\I}_2))  =\Pi$).
\end{enumerate}
\end{theorem}
Figure~\ref{FIG:HexS} illustrates all the relations in the hexagon of opposition described in Theorem~\ref{THM:HEXREL}. This figure also shows the two triangles $T_1:(s_1,s_2,s_6)$ and $T_2:(s_3,s_4,s_5)$.  We note that the sides of  $T_1$ consist of contrary relations, whereas the sides of $T_2$ consist of subcontrary relations. Moreover, 
the coherent part of the imprecise assessments defined by sentences in $T_1$
(i.e., $\D_1=\coh(\I_1)$, $\D_2=\coh(\I_2)$ and
$\D_3=\coh(\I_6)$)
forms  a tripartition $(\mathcal{D}_1,\mathcal{D}_2,\mathcal{D}_3)$ of $\Pi$. Whereas, the
 imprecise assessments defined by sentences in $T_2$ are such that
 $\coh(\I_3)=\D_1\cup \D_3$,
$\coh(\I_4)=\D_2\cup \D_3$, and
$\coh(\I_5)=\D_1\cup \D_2$.

By basing the hexagon of opposition on the square of opposition $(A(x),E(x),I(x),O(x))$ (as introduced in Section~\ref{SEC:GSQUARE}) we obtain the following hexagon of opposition: $(A(x),E(x),I(x),O(x),U(x),Y(x))$ with $x\in ]1/2,1]$, where $U(x)$ denotes $A(x)\vee E(x)$ and $Y(x)$ denotes $I(x)\wedge O(x)$ (see Table~\ref{Table:GH}). Figure~\ref{FIG:hexagon} illustrates the hexagon $(A(x),E(x),I(x),O(x),U(x),Y(x))$ with $x\in ]1/2,1]$.
\begin{table}[!h]
	\centering
	\begin{tabular}{llll}\hline \hline
		\multicolumn{2}{l}{Sentence} & Probability constr. & Assessment on $P|S$  \\
		\hline
$U(x):$& $A(x)\vee E(x)$ & 
$ \begin{array}{c}
p(P|S)\geq x\\
\text{or}\\
p(\no{P}|S)\geq x
\end{array}$
&${\mathcal{I}}_{U(x)}=[0,1-x]\cup [x,1] $\\[2.em] 
$Y(x):$& $I(x)\land O(x)$ & 
		$ \left 
		\{\begin{array}{l}
		 p(P|S)> 1-x\\
		p(\no{P}|S)>1- x
		\end{array}
		\right.$ 
		&
		${\mathcal{I}}_{Y(x)}=]1-x,x[$\\ 
		\hline        
		$U(1):$&
		$ \begin{array}{c}
		$Every $S$ is $P$$\\
		$or$\\
		$No $S$ is $P$$
		\end{array}$
		  &
		  $ \begin{array}{c}
		  p(P|S)=1\\
		  \text{or}\\
		  p(\no{P}|S)=1
		  \end{array}$
		  &
		   $\mathcal{I}_U=\{0\}\cup\{1\}$\\[2.em]
$Y(1):$&$
\begin{array}{c}
	$Some $S$ is $P$$\\
	$and$\\
	$Some $S$ is $\no{P}$$\\
\end{array}
$ 
  &		
				$ \left 
\{\begin{array}{l}
p(P|S)> 0\\
p(\no{P}|S)>0
\end{array}
\right.$ 
& $\mathcal{I}_Y=]0,1[$ \\
\hline \hline
	\end{tabular}
	\caption{Probabilistic interpretation of the sentence types  at the top ($U$) and at the bottom ($Y$) of the hexagon of opposition involving generalized quantifiers $Q$ defined by a threshold $x$ (with $x\in]\frac12,1]$) on the subject $S$ and predicate $P$ and the respective imprecise probabilistic assessments ${\mathcal{I}}_{U(x)}$, and $\mathcal{I}_{Y(x)}$  on the conditional event $P|S$ (above). When $x=1$, we obtain our probabilistic interpretation of the  traditional sentence types $U$, $Y$.}
	\label{Table:GH}
\end{table}

\begin{figure}
\begin{center}
\ifx\JPicScale\undefined\def\JPicScale{1.4}\fi
\psset{unit=\JPicScale mm}
\psset{linewidth=0.3,dotsep=1,hatchwidth=0.3,hatchsep=1.5,shadowsize=1,dimen=middle}
\psset{dotsize=0.7 2.5,dotscale=1 1,fillcolor=black}
\psset{arrowsize=1 2,arrowlength=1,arrowinset=0.25,tbarsize=0.7 5,bracketlength=0.15,rbracketlength=0.15}
\begin{pspicture}(25,0)(95,85)
\psbezier[linewidth=2pt]{<-}(25,31.88)(25,31.88)(25,50.62)(25,57.5)
\psline[linewidth=2pt]{<-}(95,31.88)(95,57.5)
\psline[linewidth=2pt]{<-}(31.25,21.25)(53.75,8.75)
\psline[linewidth=2pt]{<-}(66.25,81.25)(88.75,68.75)
\psline[linewidth=2pt]{->}(31.25,68.75)(53.75,81.25)
\psline[linewidth=2pt]{->}(66.25,8.75)(88.75,21.25)
\rput(95,62.5){$E(x): p(P|S)\leq 1-x$}
\rput(60,85){$U(x):p(P|S)\in [0,1-x]\cup [x,1]$}
\rput(25,63.12){$A(x):p(P|S)\geq x$}
\rput(25,25.62){$I(x):p(P|S)>1-x$}
\rput(60,5){$Y(x):1-x < p(P|S)<x$}
\rput(95,25.62){$O(x):p(P|S)<x$}
\psline[linewidth=2pt,linestyle=dashed,dash=1 1,linecolor=blue](42.5,62.5)(77.5,62.5)
\psline[linewidth=2pt,linestyle=dotted,linecolor=red](42.5,25)(77.5,25)
\psline[linewidth=2pt,linestyle=dashed,dash=1 1,linecolor=blue](35,30.62)(85,55.62)
\psline[linewidth=2pt,linestyle=dotted,linecolor=red](35.62,30)(85.62,55)
\psline[linewidth=2pt,linestyle=dashed,dash=1 1,linecolor=blue](35.62,55)(85.62,30.62)
\psline[linewidth=2pt,linestyle=dotted,linecolor=red](35,54.38)(85,30)
\psline[linewidth=2pt,linestyle=dotted,linecolor=red](59.5,81.1)(59.5,8.3)
\psline[linewidth=2pt,linestyle=dashed,dash=1 1,linecolor=blue](60.4,81.1)(60.4,8.6)
\psline[linewidth=2pt,linestyle=dashed,dash=1 1,linecolor=blue](29.8,55)(57,8.9)
\psline[linewidth=2pt,linestyle=dashed,dash=1 1,linecolor=blue](89.9,55.1)(63.3,8.8)
\psline[linewidth=2pt,linestyle=dotted,linecolor=red](56.5,81)(29.9,31.8)
\psline[linewidth=2pt,linestyle=dotted,linecolor=red](63.6,81)(89.7,31.4)
\end{pspicture}

\end{center}
\caption{\label{FIG:hexagon} Probabilistic hexagon of opposition defined on the six sentence types  with the threshold $x\in]\frac12,1]$ (see also Table~\ref{Table:GS}). It provides a new interpretation of the hexagon of opposition, which we compose of the probabilistic square of opposition and the two additional vertices $U(x)$ (i.e., $A(x) \vee E(x)$; top)  and $Y(x)$ (i.e., $I(x) \wedge O(x)$; bottom). For the meaning of the lines see Figure~\ref{fig:QsquareS}.}
\end{figure}
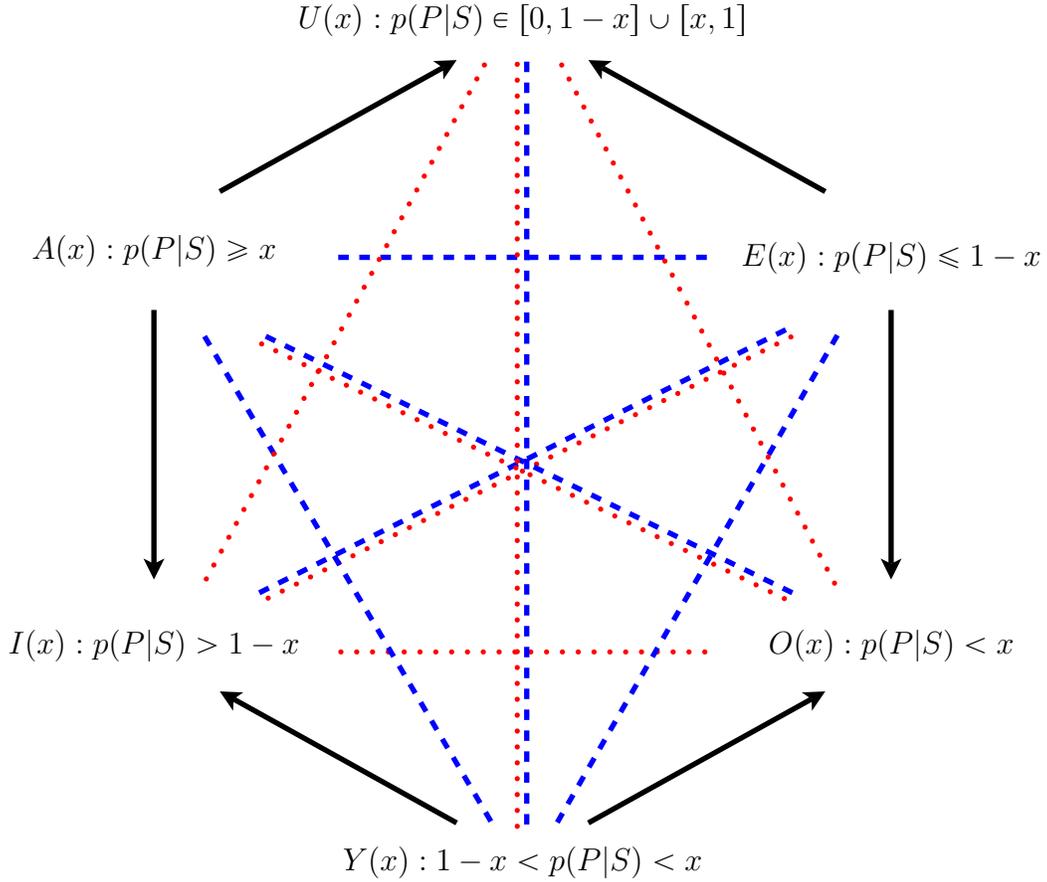
We now consider a generalization of  the  hexagon of opposition  $(A(x),E(x),I(x),O(x),U(x),Y(x))$ by considering $n$ conditional events. In particular, let $\mathcal{F}=(P_1|S_1,\ldots,P_n|S_n)$ be a sequence of $n$ conditional events. Exploiting Remark~\ref{REM:HEXTRIP}, we construct a hexagon of opposition by considering the following tripartition of $[0,1]^n$: $(\mathcal{B}_1(x),\mathcal{B}_2(x), \mathcal{B}_3(x))$, with $x\in]1/2,1]$, where
\[
\begin{array}{ll}
\mathcal{B}_1(x)=\{(p_1,\ldots,p_n)\in[0,1]^n:\sum_{i=1}^n \frac{p_i}{n}\geq x\},\\
\mathcal{B}_2(x)=\{(p_1,\ldots,p_n)\in[0,1]^n:\sum_{i=1}^n \frac{p_i}{n}\leq 1- x\},\\
\mathcal{B}_3(x)=\{(p_1,\ldots,p_n)\in[0,1]^n:1- x<\sum_{i=1}^n \frac{p_i}{n}<x \}.\\
\end{array}
\]
We obtain the following (generalized) hexagon of opposition  $(A(x),E(x),I(x),O(x),U(x),Y(x))$, with the quantified statements $A(x):(\F,\mathcal{I}_{A(x)})$, $E(x):(\F,\mathcal{I}_{E(x)})$, $I(x):(\F,\mathcal{I}_{I(x)})$, $O(x):(\F,\mathcal{I}_{O(x)})$, $U(x):(\F,\mathcal{I}_{U(x)})$, $Y(x):(\F,\mathcal{I}_{Y(x)})$,
where
\[
\begin{array}{ll}
\mathcal{I}_{A(x)}=\mathcal{B}_1(x),\; 
\mathcal{I}_{E(x)}=\mathcal{B}_2(x),\; \mathcal{I}_{Y(x)}=\mathcal{B}_3(x),\\
\mathcal{I}_{I(x)}=\mathcal{B}_1(x) \cup \mathcal{B}_3(x)=\{(p_1,\ldots,p_n)\in[0,1]^n: \sum_{i=1}^n \frac{p_i}{n}>1-x\},\\
\mathcal{I}_{O(x)}=\mathcal{B}_2(x)\cup \mathcal{B}_3(x)=\{(p_1,\ldots,p_n)\in[0,1]^n: \sum_{i=1}^n \frac{p_i}{n}<x\},\\
\mathcal{I}_{U(x)}=\mathcal{B}_1(x)\cup \mathcal{B}_2(x)=\{(p_1,\ldots,p_n)\in[0,1]^n: \sum_{i=1}^n \frac{p_i}{n}\geq x \text{ or } \sum_{i=1}^n \frac{p_i}{n}\leq 1-x\}.
\end{array}
\]

 \section{Concluding Remarks}\label{SEC:CONCL}

Finally, we note that conditional probability interpretations of quantified statements were also proposed in psychology (see, e.g.,  \cite{cohen99,cohen12,oaksford07,pfeifer06,pfeifer13b,pfeifer05a,pfeiferTulkkiSub}), since
generalized quantifiers  are  psychologically much more  plausible compared to 
the traditional logical quantifiers, as the latter are either too strict  ($\forall$ does not allow for exceptions) or too weak ($\exists$ quantifies over at least one object) for formalizing everyday life sentences. Recent experimental data suggests that people negate conditionals and quantified statements mainly by  building contraries (in the sense of  inferring $p(\neg C|A)=1-x$ from the negated $p(C|A)=x$) but hardly ever by building contradictories (in the sense of  inferring $p(C|A)<x$ from the negated $p(C|A)=x$; see \cite{pfeifer12x,pfeiferTulkkiSub}). However, this empirical result calls for further experiments. The square   presented  in Section~\ref{SEC:GSQUARE} and the hexagon presented in Section~\ref{SEC:HEX} can serve as a new rationality framework for formal-normative and psychological investigations of  basic relations among quantified statements.

\section*{Acknowledgement}
 
We thank \emph{Deutsche Forschungsgemeinschaft} (DFG),  \emph{Fondation Maison des Sciences de l'Homme} (FMSH), and \emph{Villa Vigoni} for supporting joint meetings at Villa Vigoni where parts of this work originated (Project: ``Human Rationality: Probabilistic Points of View''). Niki Pfeifer is supported by his DFG project PF~740/2-2 (within the SPP1516 ``New Frameworks of Rationality'').  Giuseppe Sanfilippo is supported by the INdAM--GNAMPA Project (2016 Grant U 2016/000391).


\end{document}